\documentclass{my-dnce}
\usepackage{my-dnce}

\Page{1}       

\usepackage[T1]{fontenc}
\usepackage{ae,aecompl}

\theoremstyle{definition}  
\theoremstyle{remark}      
\theoremstyle{plain}       
\newtheorem{theorem}{Theorem}
\newtheorem{proposition}[theorem]{Proposition}

\theoremstyle{remark}

\theoremstyle{definition}

\titlemark{\ A. Galindro, A. Cerveira, D.F.M. Torres, J. Matias and A. Marta-Costa \ }

\begin{document}

\title{\first{A Mathematical Model for Vineyard Replacement 
with Nonlinear Binary\\ Control Optimization}}

\setcounter{footnote}{1}

\author{\noindent\large An\'{\i}bal Galindro$^{1}$\footnote{Corresponding author.\\
Email address: anibalg@utad.pt}~, Adelaide Cerveira$^{2}$, Delfim F. M. Torres$^{3}$, 
Jo\~{a}o Matias$^{4}$, Ana Marta-Costa$^{1}$}


\address{\normalsize $^{1}$Centre for Transdisciplinary Development Studies, 
University of Tr\'as-os-Montes and Alto Douro,\\[-2mm] 
Polo II--ECHS, Quinta de Prados, 5000-801 Vila Real, Portugal.\\
$^{2}$INESC-TEC, Department of Mathematics, 
University of Tr\'as-os-Montes and Alto Douro,\\[-2mm] 
Quinta de Prados, 5000-801 Vila Real, Portugal.\\
$^{3}$Center for Research and Development in Mathematics and Applications (CIDMA),\\[-2mm] 
Department of Mathematics, University of Aveiro, 3810-193 Aveiro, Portugal.\\
$^{4}$CMAT-UTAD, Department of Mathematics, University of Tr\'as-os-Montes and Alto Douro,\\[-2mm] 
Quinta de Prados, 5000-801 Vila Real, Portugal.}


\abstract{
\begin{table}[ht!]
\vspace*{-5mm} \doublerulesep 0.05pt \tabcolsep 7.8mm \vspace*{2mm}
\setlength{\tabcolsep}{7.5pt}
\hspace*{-2.5mm}\begin{tabular*}{171.45mm}{r|||||l}
\multicolumn{2}{l}{\rule[-6pt]{171.45mm}{.01pt}}\\
\parbox[t]{6cm}{\small
\vspace*{.5mm}
\hfill {\bf Submission Info}\par
\vspace*{2mm}
\hfill Submitted 30 May 2018 \par 
\hfill Revised 28 July 2018  \par
\hfill Accepted 6 Sept 2018\par
\noindent\rule[-2pt]{6.3cm}{.1pt}\par
\vspace*{2mm}
\hfill {\bf Keywords}\par
\vspace*{2mm}
\hfill Mathematical Modelling\par
\hfill Binary Optimization\par
\hfill Vineyard Replacement\par
\hfill Douro Wine Region}
&
\parbox[t]{10.5cm}{
\vspace*{.5mm}
{\normalsize\bf Abstract}\par
\renewcommand{\baselinestretch}{.8}
\normalsize \vspace*{2mm} {\small Vineyard replacement is a common practice 
in every wine-growing farm since the grapevine production decays over time 
and requires a new vine to ensure the business sustainability. In this paper, 
we formulate a simple discrete model that captures the vineyard's main dynamics 
such as production values and grape quality. Then, by applying binary non-linear 
programming methods to find the vineyard replacement trigger, 
we seek the optimal solution concerning different governmental 
subsidies to the target producer.}
\par
\par
\par
\hfill{\ }}\\[-4mm]
&\\
\multicolumn{2}{l}{\rule[15pt]{171.45mm}{.01pt}}\\\end{tabular*}
\vspace*{-7mm}
\end{table}
}

\maketitle
\thispagestyle{fancy}


\renewcommand{\baselinestretch}{1}
\normalsize


\section{Introduction}

\noindent Laying in the south western part of Europe, a significant portion of 
Portugal's economy substantially relies on prolific tourism endeavours 
and solid agricultural systems. It is fair to say that the wine sector is one 
of the most important Portuguese agricultural sub-sectors, mainly due to its solid 
traditional background, as bulky statistics display on official records. 
In fact, the International Organisation of Vine and Wine (OIV) caps Portugal 
as the $11^{th}$ greatest vineyard area and the ninth wine exporter (in value) 
worldwide OIV \cite{1}. Even though these numbers don't seem too impressive, 
they represent a convincing effort taking into account that Portugal's territory 
is quite limited. Apart from the pragmatic worldwide rankings, the Portuguese vineyards 
deliver way more than a simple grape output since the oldest demarcated region 
was classified by UNESCO as a World Heritage in 2001. The country is well aware 
of this valuable heritage and the past few years revealed a set of meaningful 
efforts aiming the overall improvement of this sector. These attempts are ranged 
from direct investment (vineyard renewal, infrastructures) to academic oriented 
studies that intend to attain higher sustainability \cite{2} or efficiency levels \cite{3}. 
Aside from the more immutable orographic and edaphoclimatic factors, a vineyard holds 
a very complex set of variables that may be optimized in order to attain a more prolific output. 
The scope of our work is focused primarily on the vineyard replacement timing 
and inner multi-plot optimization. Therefore, the extensive set of components 
that characterize the entire vineyard system are omitted here. 

There is no such thing as a static agricultural system: every crop needs fertilization, 
rotation or replacement. Vineyards are not an exception whatsoever. 
The core subject of this article falls deeply into a pure optimization effort: 
we intend to provide optimal multi-plot vineyard replacement 
in a given time-frame. For that a mathematical model is proposed.
Several optimization tools to approach the problem were evaluated. 
Binary non-linear optimization methods 
were selected due to their suitability. 
    
The remainder of this article is organized as follows.
Section~\ref{LR} gives account of previous work, while Section~\ref{data} 
presents problem's function fitting preceded by a brief description 
of the collected data. Our model and methodology is described in 
Section~\ref{methodology}, followed by the presentation of the numerical 
results and discussion in Section~\ref{results}.
Finally, Section~\ref{conclusions} gives the main conclusions of our article.


\section{Literature Review}
\label{LR}

\noindent It is a fair assumption to state that farmers, traditionally, 
adapt their judgement in agricultural planning  based on their previous 
experience \cite{4}. However, the last century prolific innovation increased 
the agricultural productivity massively, making this sector a contributor 
to the overall economic development of several countries \cite{5}. Even though 
the simplistic and traditional approaches from casual familiar farms are usually 
completely outclassed (productively) by the huge well-structured and highly 
technological explorations, there is still room to the small producers that 
are willing to enter the market and attain a marginal revenue from the activity \cite{6}. 
As stated on previous section, available bibliography compile several and 
distinctive optimization attempts, generally adapted to the agricultural 
subset that the authors intend to study. Mixed-integer programming planning (MIP) 
can be found in models developed by authors such as Masini \cite{7}, 
with an application to the fruit industry. Troncoso and Garrido \cite{8} 
applied MIP to forestry productions with logistic optimization, while
Fonseca, Cerveira and Mota \cite{9} developed a forest thinning 
and clear-cutting model considering a forthcoming five year planning horizon. 
Jena and Poggi \cite{10} adopt an MIP model to study and identify 
an harvest planning in the Brazilian sugar cane industry. Regarding 
specifically the wine sector, Ferrer et al. \cite{11} apply MIP 
to obtain an optimal schedule for wine grape harvesting. Because, 
frequently, the agricultural systems intend to maximize/minimize 
several variables simultaneously (e.g., a farmer may intend 
to maximize production while minimizing pollution or environmental impact), 
the multi-objective optimization approach is quite common \cite{12}. 
Groot et al. \cite{13} apply multi-objective optimization and prove 
its suitability in mixed farm design, accounting several production 
and environmental variables. Agricultural subsistence is also highlighted 
on the multi-objective optimization models of Klein et al. \cite{14} 
and Banasik et al. \cite{15}, which respectively intend to adapt 
the farm system to climate changes and evaluate closed loops on supply 
chains.\footnote{Closed-loop supply chains are often defined 
with integrated management of processes that aim to reintroduce 
returned products, parts and materials into the supply chain.} 
The willingness to optimize a given system can also be attained 
with more straightforward approaches. Millar et al. \cite{16} 
and Garc\'{\i}a-D\'{\i}az et al. \cite{17} use sampling methods, 
index formulation and comparative analysis to pheromone baited-traps 
and soil carbon optimization, respectively. Patakas, Noitsakis and Chouzouri \cite{18} 
analyse the relationship between vine transpiration and water stress 
in order to optimize irrigation patterns. 
Atallah et al. \cite{19} and Ricketts et al. \cite{20} put their efforts into 
the optimal grapevine disease control in New York and California vineyards, respectively. 
The optimal control (OC) approach bundles multi-period optimization of the dynamic programming 
with the cutting edge suitability to treat a large spectrum of problems \cite{21}. 
The OC methodology became quite popular with applications in several areas, such as Economics \cite{22,22b}, 
Biology and Health related problems \cite{23,24}. Previous works also blend the OC with other 
methodologies, such as the Maximum-Entropy estimation \cite{25} 
and Deep Learning approximation \cite{26}. Since OC is quite adaptable, 
this theory has also been applied to agricultural endeavours, e.g., 
to overcome invasive species \cite{27}, to control pests in an optimal way \cite{27b},
or to provide real-time greenhouse heating \cite{28}. Logically, since grape growing 
is a subset of the agriculture endeavour, it is expected to see at least a few studies 
of OC dedicated, exclusively, to vineyard challenges. Surprisingly, to the best of our knowledge, 
Schamel and Schubert \cite{29} presents the lonely attempt to bundle OC and vineyard/grapevine systems. 
They develop an OC model to attain the ideal crop thinning in viticulture, 
maximizing grape quality and quantity \cite{29}. Moreover, they give a brief overview of 
optimization methods, noting the broad range of applications to the several subsets of the agricultural entrepreneurship. 
In contrast, our problem seeks a purely binary non-linear programming (BNLP) approach. Even though a few authors, 
such as Silva, Marins and Montevechi \cite{30} or Arredondo-Ram\'{\i}rez et al. \cite{31}, 
actually applied BNLP to the sugar production and the design of water systems in agriculture, 
their problem formulation is still quite different from our own.         


\section{Collected data and function fitting}
\label{data}

\noindent Before getting to the concrete model, this section describes briefly the available 
database and precedent grape quality/quantity function fitting. A face-to-face survey 
was conducted to compile extensive information related to almost every production input 
factor and coming output and revenues. Unfortunately, there are weighty difficulties gathering 
the data. Sometimes, the farmers are not available to provide an accurate response or, 
in the worst case scenario, they do not provide an answer at all. We were able 
to obtain 20 surveys exclusively from the Douro region. These surveys compile bulky 
farm information, from which we only extract four variables: 
vineyard with each inherent plot age and area, 
grape production (in tons), and revenues (in Euro). The age of each vineyard plot 
is irrefutably the key factor of the whole study, because we intend to investigate 
the optimal vineyard replacement age. To fulfil the analysis requirements, 
one needs to associate the grape quality and output per each vineyard age. 
In fact, the combined grape quality and quantity are directly correlated 
to the final revenue of the sampled farm. The main goal is to optimize the 
vineyard replacement timing, accounting that the owner intends to maximize their profit 
on a previously ranged (5, 10, 15 and 60 years) decision time-frame. To accomplish 
that purpose, we need to fit two functions that depend on the vineyard's age: 
Quality ($Q$) and Quantity ($K$). To achieve a trustworthy representation 
of the expected production (in Kg per hectare) for each vineyard age, 
we presume that this function should follow a concave down display. 
Higher forms of nonlinearity could be attempted to better fit the data,  
but it is well known that the grape production rises until a certain amount, 
which may also be steady for a while, and after some age it starts decaying. 
On the other hand, the quality of the grape is expected to steadily rise 
following a simple incremental function equation -- see \eqref{eq:quality}. 
Therefore, the polynomial degree was chosen to describe observable 
features of the vineyard and function fitting.
According to the survey information, it is assumed that the vineyard productivity usually 
rises until a certain age, starting to decay at some point, when the vineyard gets older. 
It is also known that a vineyard does not start its production, from the commercial point of view, 
until it gets 5 years old. Therefore, to penalize our function, a few young vineyards 
with production value of zero were introduced.
We fitted our data with a 2nd degree polynomial function using \textsf{Matlab}, 
obtaining a reasonable goodness of fit, precisely, a $R$-square of 0.8051 (see Appendix). 
Figure~\ref{fig:1} shows the fitted function and the actual productivity data.
\begin{figure}[ht!]
\centering
\includegraphics[scale=0.5]{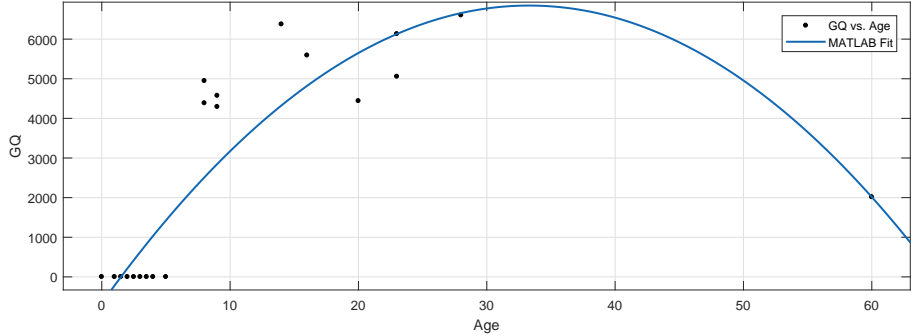}
\caption{Expected vineyard production per hectare (continuous--blue).}
\label{fig:1}
\end{figure}

This work assumes, by acquainting the collected surveys, that the grape quality varies 
along the time span. To get a hint about grape quality, the aforementioned data 
(production per hectare, Productivity) was recycled and linked with the total revenues 
from each farm. Equation \eqref{eq:rev} expresses the resulting grape quality variable ($GQ$):  
\begin{equation}
\label{eq:rev}
GQ = \frac{TotalRevenues}{Productivity}.
\end{equation} 
Following the same premise of the grape quantity per age fitting function, 
it is assumed that the grape quality displays an expected pattern based 
on the collected surveys. This time we infer that the grape quality should linearly 
rise as the vineyard gets older. Performing an Ordinary Least Square regression in 
\textsf{Python} (see Appendix), to relate the quality proxy $GQ$ with the vineyard age, 
we also applied a bootstrap re-sampling of 500 simulations to estimate our standard errors, 
confidence intervals, and to provide a graphical output: see Figure~\ref{fig:2}.
\begin{figure}[ht!]
\centering
\includegraphics[scale=0.5]{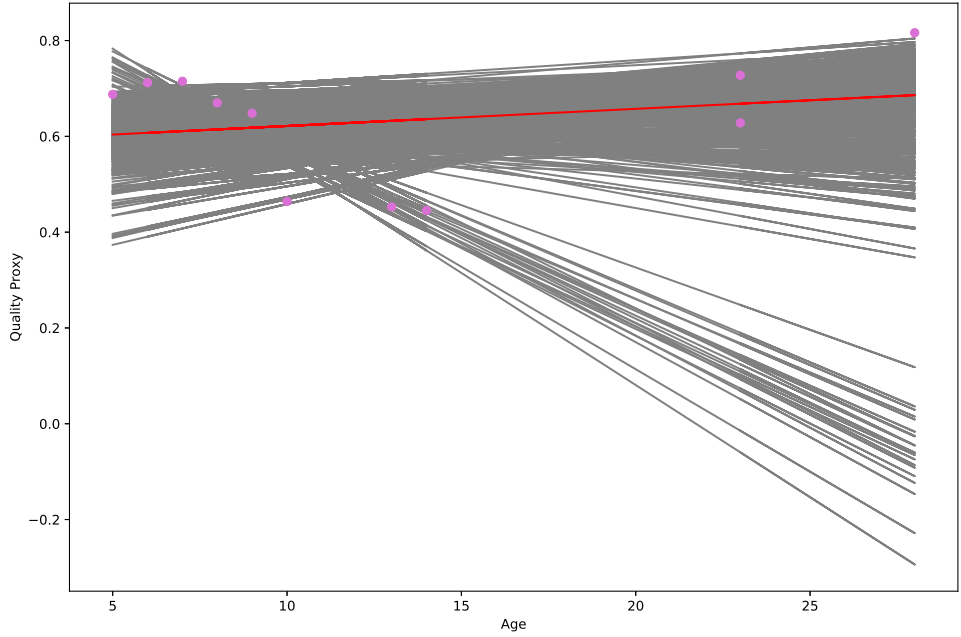}
\caption{Expected grape-quality improvement (continuous--red) obtained 
with 500 resampled bootstrap simulations (continuous--grey).}
\label{fig:2}
\end{figure} 

The limited number of observations and possibly biased data 
(sometimes farmers do not provide accurate information) 
certainly harms our OLS (ordinary least squares) regression results (see Appendix). 
It is unavoidable to state that the obtained results in Appendix
are not perfect, compiling not significant coefficients $\beta_1$ 
and a poor R-squared value. However, it is not our intention to perform 
a model that arbitrarily assumes values. Therefore, the calculated $\beta_1$ 
provides the least of two evils and is used in the next section. 
Quality $Q$ and quantity $K$ functions are computed according to 
\eqref{eq:quality} and \eqref{eq:quantity}, respectively:
\begin{equation}
\label{eq:quality}
Q(t) = 0.0036 I_t,
\end{equation}
\begin{equation}
\label{eq:quantity}
K(t) = -661.4+451.1 I_t-6.774I^{2}_t,
\end{equation}
where $I_t$ represents the vineyard age at time $t$. 
Since the vineyard produces once a year and a continuous decision time-frame 
is not realistic from producer's point of view, 
the problem formulation acquaints a discrete model.
It is worth to note that the intercept of the regression 
is omitted in \eqref{eq:quality}, because we are only interested 
in the evolution of the quality upon vineyard's age.


\section{Methodology}
\label{methodology}

\noindent In order to evaluate a real situation, we have selected a sample farm in the Douro region 
that relies its production on 5 different plots with the correspondent approximate area (in $ha$) 
and initial age per plot given by $A=[4.47,1.45,0.44,1.66,0.5]$ and $I_0=[20,30,11,5,58]$, respectively.   
Let $T=\{0, \ldots, 59\}$ be the planning horizon with 60 one-year periods, where $t=0$ 
corresponds to the first year and $P=\{1,\ldots,5\}$ is the set of plots. Assigned to each plot 
$j \in P$, we consider the following parameters: $Ii_j$, the age of the stand in the first period; 
$A_j$, the area of the stand in $ha$. 

Considering equations~\eqref{eq:quality} and \eqref{eq:quantity}, in our investigations
the following constant values are considered:
\begin{itemize}
\item $qc=0.0036$, the quality constant;
\item $p_0=-661.4$;
\item $p_1=451.1$;
\item $p_2=-6.774$;
\item $Pu=3$, the selling price of the grapes per $Kg$;
\item $S=10000$, cost of replacing the vine in Euro per hectare.
\end{itemize}
We also consider the following variables: $u_j^t$, binary variables taking value 1 if a cutting  
is performed in the plot $j \in P$ in the given period $t \in T$ and 0 otherwise; 
and $Id_j^t$, integer variables represent the age of the plot $j$ in period $t$. 
We intend to maximize \eqref{eq:fo}, which is our model's objective function,
\begin{equation}
\label{eq:fo}
\displaystyle \max \quad  \sum_{t \in T} \sum_{j \in P} \left( 
Pu \cdot qc \cdot Id_j^t \cdot A_j \cdot (p_2 \cdot (Id_j^t)^2 
+ p_1 \cdot Id_j^t+p_0)-S \cdot A_j \cdot u_j^t \right),
\end{equation}
subject to the constraints 
\begin{gather}
Id_j^t=t+Ii_j-\sum_{\ell=0}^{t-1} u_j^\ell \cdot (Id_j^{\ell}+1), 
\quad j \in P, \quad t \in T, \label{eq:id}\\
Id_j^t \in \mathbb{Z}_0^+, \quad  j \in P, \quad t \in T, \label{eq:Id_integer}\\ 
u_j^t   \in \{0,1\}, \quad  j \in P, \quad t \in T, \label{eq:u_binary}
\end{gather}
where \eqref{eq:id} establish the age of each stand $j$ at each period $t$ 
and constraints \eqref{eq:Id_integer} and \eqref{eq:u_binary} acquaint the variables domain.
The objective function \eqref{eq:fo} compiles on the left positive size 
the variables that influence the farm's revenue, such as the grape price ($Pu$) 
and the age dependent grape quality ($qc$), alongside the production values 
given by the polynomial function $K(t)$ underpinned by the Area ($A$). 
The negative right-hand side discloses the expenses (augmented by the Area $A$) 
when the vine is cut. Therefore, the objective function \eqref{eq:fo} aims 
to maximize the farm profit in a given time-frame, optimizing the ideal timing of cutting ($u$).

Since the plots are not correlated with each other, the optimal replacement solution 
can be attained solving the general average profit function $L(t)$, given by \eqref{eq:L},
for a generic single plot:
\begin{equation}
\label{eq:L}
L(t) = K_{Id(t)} \cdot Q_{Id(t)} \cdot A - s \cdot A,
\quad t\in[0,+\infty[ ,
\end{equation}
with
\begin{equation*}
K_{Id(t)} = Pu \cdot qc \cdot Id_t, 
\quad 
Q_{Id(t)} = p_2 \cdot (Id_t)^2 + p_1 \cdot Id_t+p_0.
\end{equation*}
Function \eqref{eq:L} heads up the farming revenues 
minus the fixed costs of renewing the vineyard:
\begin{eqnarray}
Profit \ in \ the \ year \ 'n' 
\rightarrow L'(Id_n) = K_{Id(n)} \cdot Q_{Id(n)} \cdot A,
\end{eqnarray}
\begin{eqnarray}
\label{pf}
Profit \ in \ the \ cycle \ of \ 'N' \ size \rightarrow \hat{L}(N)
=\frac{\sum_{i=0}^{N}L(i) - s \cdot A}{N}.
\end{eqnarray}

\begin{proposition}
\label{prop:01}
The solution is given by $N=59$ when we look to the value 
of $N$ that maximizes $\hat{L}(N)$. 
Replacing a generic vineyard at their 59 years of age, that is, $Id_t=59$,
gives the solution to the optimization problem in an infinite time horizon.  
\end{proposition}

\begin{proof}
\label{proof1}
Consider a single plot. Assume
\begin{eqnarray*}
P(Id_t) = Pu \cdot qc \cdot Id_t \cdot A 
\cdot \left(p_2 \cdot (Id_t)^2 + p_1 \cdot Id_t+p_0\right) 
\end{eqnarray*}
and
\begin{eqnarray*}
c \cdot u_t = S \cdot A \cdot u_t. 
\end{eqnarray*}
Being $\sum_{i=0}^{59} u^{*}_i=1$ the optimal replacement control for a single cut, 
we want to prove that for all $I_t$, $t\in[0,59]$, does not exist
$\sum_{i=0}^{59} u'_i \geq 2$ such that
\begin{eqnarray*}
P(Id^{*}_t) - c \sum_{i=0}^{59} u^{*}_t >  P(Id_t) - c \sum_{i=0}^{59} u'_t,
\end{eqnarray*}
\begin{eqnarray*}
P(Id^{*}_t) >  P(Id_t) - (b-1) \cdot c, \text{ with } \ b \in \{2,3,\ldots,60\},
\end{eqnarray*}
where $L(Id^{*}_t)$ acquaints a full 59 year optimal cycle:
\begin{equation*}
L(Id^{*}_t) + L(Id_S) >  L(Id^{*}_t) + L(Id_S)  - (b-1) \cdot c, 
\text{ with } \ S \in \{0,1,2,\ldots,59\},
\quad V \in \{0,1,2,\ldots,59\},
\end{equation*} 
\begin{eqnarray*}
L(Id_S) > L(Id_V)  - (b-1) \cdot c.
\end{eqnarray*}
Assuming that $L(Id_S)=\min\{L(Id_t)\}$ and $L(Id_V)=\max\{L(Id_t)\}$, 
which is the maximum revenue range among two farm status, we have
\begin{eqnarray*}
0 > L(Id_V)-L(Id_S)  - (b-1) \cdot c.
\end{eqnarray*}
Replacing with real values, and because the minimal $b$ is $b=2$, we get
\begin{eqnarray*}
0 > 2885.66-(-2.34) - 10000, 
\end{eqnarray*}
\begin{eqnarray*}
0 > -7112.
\end{eqnarray*}
The proof is complete.
\end{proof}


\section{Results and Discussion}
\label{results}

\noindent The numerical simulations were performed in the 
FICO Optimization Xpress software with the appended code 
written in \textsf{Mosel} \cite{Mosel:32} (see Appendix).
Due to the high non-linearity of the model, underpinned by the polynomial in the 
objective function and the variables $Id^{t}_j$ and $u^{t}_j$, the software execution time 
exponentially rises with the time-frame. The maximum horizon $t\in[0,59]$ 
optimal solution was performed iteratively for a single cut $\sum_{i=0}^{t} u_i=1$ 
or none $\sum_{i=0}^{t} u_i=0$ (61 combinations) and validated with Proposition~\ref{prop:01}.
The shorter time-frame optimizations by the farmer were performed directly with the FICO Optimization Xpress 
software with stepwise simulations until the final time $t=59$, changing the model initial conditions at each step
(see Appendix). The producers total yield and plot replacement age, per each considered time-frame, 
are stated on Table~\ref{table1}. 
\begin{table}[ht!]
\doublerulesep 0.1pt
\tabcolsep 7.8mm
\centering
\caption{\rm Vineyard replacement age per plot, for each decision time-frame, 
and the resulting total yield in Euro.}
\vspace*{2mm}
\renewcommand{\arraystretch}{1.3}
\setlength{\tabcolsep}{20pt}
\footnotesize{\begin{tabular*}{16.5cm}{ccccccc}
\hline\hline\hline
 Age & Plot 1 & Plot 2 &Plot 3 & Plot 4 & Plot 5 & Total Yield \\
\hline
5 Years & 70 & 70 & None & None & 73 & 691238,21 \\
10 Years & 71 & 71 & None & None & 69 & 686398,61 \\
15 Years & 59 & 60 & 66 & None & 64 & 782085,19  \\
60 Years & 61 & 44 & None & None & 58 & 793114,13  \\
Infinite & 59 & 59 & 59 & 59 & 59 & 755712,99 \\
\hline\hline\end{tabular*}}
\renewcommand{\arraystretch}{1}
\label{table1}
\end{table}
As expected, when the producer is not able to foresee, in a sufficiently large time-frame decision, 
the overall yield is significantly smaller than in a broad optimization period. 
The difference between optimizing in a 15 years time-frame and 5 or 10 years is quite obvious 
since the larger period outperforms the shorter ones in roughly 100 000 Euro over the 60 year analysis. 
The optimal global solution (60 years) caps only a marginally better result than the stepwise 15 year analysis, 
while the infinite horizon solution (IHS) acquaints a sensibly average result. The IHS operates naively, 
cutting the vineyard every time it reaches the age of 59 years old, maximizing the average profit function. 
That routine does not consider the ending of the simulation, performing unessential vineyard replacements. 
Nonetheless, if the producer does not have a sufficiently large planning time-frame, 
it is better to rely his decisions on the IHS than short term optimizations. Another interesting 
result, is that the shorter time-frame of 5 years planning, slightly outperforms the 10 years solution. 
This result might be explained by the circumstantial initial conditions of the model. 

Even though we find this problem interesting from the producers point of view, 
we want to evaluate this problem considering the governmental available policies. 
It is well known that the Portuguese government performs regular support to the
Douro grape producers, usually providing a beneficial price and allowing the producer 
to sell their grapes at higher prices. There has also been recent policies intended 
to provide support for vineyard replacement efforts. 

We now assume that the government fully supports every vineyard replacement 
from the producer. Sub-sequentially, \eqref{pf} is solved without the associated 
fixed costs $s A$. We gather a new optimal infinite horizon solution with the ideal 
replacement age at 49 years old, which yields a better average profit due to the absence 
of fixed and improved replacement endeavors. The results of Table~\ref{table2} display 
the trivial producer improvement (A) due to the governmental aid upon the helpless (B), 
considering the IHS average yield (Avg Yield), average replacement costs (Avg RC), 
average production (Avg Production) and, finally, the average governmental support.  
\begin{table}[ht!]
\doublerulesep 0.1pt
\tabcolsep 7.8mm
\centering
\caption{\rm Average values comparison (in Euro and Kg, Production)
when the government provides, or not, free vineyard replacement.}
\vspace*{2mm}
\renewcommand{\arraystretch}{1.3}
\setlength{\tabcolsep}{22pt}
\footnotesize{\begin{tabular*}{16.5cm}{cccccc}
\hline\hline\hline
IHS type & Avg Yield & Avg RC & Avg Production & Avg Gov Support \\ \hline
(A) IHS - NRC & 13633 & 1738 & 42316 & 1738\\
(B) IHS - RC & 13027 & 1444 & 40985 & None\\
\hline\hline\end{tabular*}}
\renewcommand{\arraystretch}{1}
\label{table2}
\end{table}
The goal is to find the beneficial grape selling price ($Pu+a$, with $a$ being the benefit) 
that the government needs to provide in order to make the producer (B) match the average yield 
of producer (A). We note that the producers (A) and (B) are the same, since they both 
represent the sampled farm from Section~\ref{methodology}. The only difference among them 
is the optimization criteria. Regarding the producer (B) production, which remains the same, 
the government needs to provide an additional 0.1257 cents per Kg to make this producer match the other one. 
\begin{table}[ht!]
\doublerulesep 0.1pt
\tabcolsep 7.8mm
\centering
\caption{\rm Matched average yield from the producer with two different governmental policies.}
\vspace*{2mm}
\renewcommand{\arraystretch}{1.3}
\setlength{\tabcolsep}{22pt}
\footnotesize{\begin{tabular*}{16.5cm}{cccccc} \hline\hline\hline
IHS type & Avg Yield & Avg RC & Avg Production & Avg Gov Support \\ \hline
(A) IHS - NRC & 13633 & 1738 & 42316 & 1738 \\
(B') IHS - RC & 13633 & 1468 & 41343 & 5167 \\
(B) IHS - RC & 13633 & 1444 & 40985 & 5151 \\ \hline\hline
\end{tabular*}}
\renewcommand{\arraystretch}{1}
\label{table3}
\end{table}  

As stated on Table~\ref{table3}, the governmental costs of this policy are way higher 
(almost triple) than the vineyard replacement support. Assuming that the producer 
previously knows about the governmental benefit, his IHS \eqref{pf} shifts from 
the previous optimal to replace the vineyard a year earlier (58 years, labeled Producer (B')), 
but his option increases the production value making the governmental support per grape (Kg) 
decrease in order to match both producers average yield. Sub-sequentially, 
the smaller benefit makes the producer (B') shift again to the original optimal 
solution at 59 years replacement age. Nonetheless, we present both solutions 
on Table~\ref{table3}, even though they don't differ too much. 

The developed model is purely deterministic: no forms of uncertainty 
were considered. Nonetheless, adding a stochastic representation to the model 
might deliver a more realistic approach if random phenomena, like the weather 
(temperature and rainfall) and disease proneness, were featured 
in the objective function \eqref{eq:fo}.

Considering our problem formulation, and assuming that 
the farmers revenues depend directly on the produced grape quantity and quality, 
it is suggested that the government should provide direct and unequivocal vineyard replacement 
support instead of grape selling benefit. Assuming that the producer operates rationally 
and relies on a IHS, the replacement support induces a more prolific optimization, 
while the selling grape benefits are very unlikely to force optimal and incremental decisions 
from the producers, added to the fact that the government expenditure 
severely bloats with the latter tool.  


\section{Conclusions}
\label{conclusions}

\noindent We considered vineyard replacement activity as a crucial factor within 
the grape-growers optimization portfolio. The inner farm gross profit function 
was considered, strictly dependent from the produced grape quantity and quality functions, 
per each vineyard age alongside the fixed cost of replacing the vineyard. To obtain those 
functions, and simulate numerically a concrete illustrative example, 20 surveys were 
collected in the Portuguese Douro wine region. It was concluded that the producers 
optimization horizon plays an important role in the overall business profitability, 
since shorter forecasts usually lead to smaller income. An infinite horizon solution (IHS)
was also calculated based on the average profit function \eqref{pf}, even though this 
solution, due to its naiveness (the age to replacement is fixed), deprecates when 
the optimization period is shorter. It was found that the IHS still outperforms short decision 
planning (5 and 10 years) in an overall 60 year time-frame. The important remark is that 
if the producer is not able to acquaint a sufficiently large period (at least 15 years) 
when optimizing his vineyard replacement endeavours, it is preferable to rely purely on the IHS. 

Afterwards, governmental intervention was considered on producers that follow the IHS. 
Namely, we investigated two possible supporting tools: providing a grape selling price benefit 
or free vineyard replacement to the producer. The results were compared and evaluated. 
The second tool is the prominent method, because it induces the producer to seek freely 
the vineyard replacement age, maximizing his overall production without considering the fixed costs. 
To match the producers that receive a free replacement scenario (average profit), 
the government benefit to the grape selling price reveals itself way more expensive 
and roughly innocuous upon the producer optimization decision. 

Our model and study, even though based on real data, assumes and also omits many parameters 
that could be considered when developing a vineyard replacement strategy, which may also 
vary within different producers. For further developments, the database could be extended 
alongside the input variables on the objective function, to explain the problem more accurately 
and also to increase the spectrum of governmental policies 
(e.g., beneficial labour or provide mechanization).  
There is also room to attempt new robust optimization methods such 
as RCMARS, generally applied to optimize trade-off between 
risk and return in the financial market \cite{Ozmen:2016}. 
An analogous problem formulation may allow the junction of data-mining methods. 
On the other hand, there are other chances to model the problem 
with an increasing level of complexity. The independence assumed 
in our variables make the nonlinear binary control viable. 
However, if somehow it becomes valuable to entangle the quality 
and quantity variables and introduce stochastic effects, such us 
the weather or vineyard diseases, approaches such as bang-bang 
control optimization \cite{Bryson} and stochastic optimal control 
under a switching regime \cite{Temocin:2014,Azevedo:2014} 
might be useful tools for further endeavours. 


\section*{Acknowledgements}

\noindent This work was supported by the R\&D Project INNOVINE \& WINE 
-- Vineyard and Wine Innovation Platform -- Operation NORTE-01-0145-FEDER-000038, 
co-funded by the European and Structural Investment Funds (FEDER) and by Norte 2020 
(Programa Operacional Regional do Norte 2014/2020).
Torres was supported by FCT through CIDMA, project UID/MAT/04106/2013.

The authors are very grateful to two anonymous Reviewers for several suggestions, 
questions and remarks, which helped them to improve the paper.


\section*{Appendix}
\label{appendix}

\noindent We have used \textsf{Matlab R2017a} to fit our data with a 2nd degree polynomial,
with Figure~\ref{fig:1} showing the fitted function and the real productivity data.
Follows our \textsf{Matlab} code:
\begin{verbatim}
function [fitresult, gof] = createFit1(Age, GQ)
%CREATEFIT1(AGE,GQ)
%  Create a fit.
%
%  Data for 'MATLAB Fit' fit:
%      X Input : Age
%      Y Output: GQ
%  Output:
%      fitresult : a fit object representing the fit.
%      gof : structure with goodness-of fit info.
%

[xData, yData] = prepareCurveData( Age, GQ );

% Set up fittype and options.
ft = fittype( 'poly2' );
opts = fitoptions( 'Method', 'LinearLeastSquares' );
opts.Robust = 'LAR';

% Fit model to data.
[fitresult, gof] = fit( xData, yData, ft, opts );

% Plot fit with data.
figure( 'Name', 'MATLAB Fit' );
h = plot( fitresult, xData, yData );
legend( h, 'GQ vs. Age', 'MATLAB Fit', 'Location', 'NorthEast' );
% Label axes
xlabel Age
ylabel GQ
grid on

Results of Fitting:

Linear model Poly2:
f(x) = p1*x^2 + p2*x + p3
Coefficients (with 95% confidence bounds):
p1 =      -6.774  (-8.864, -4.685)
p2 =       451.1  (332.2, 570)
p3 =      -661.4  (-1674, 351.1)

Goodness of fit:
SSE: 2.678e+07
R-square: 0.8051
Adjusted R-square: 0.7822
RMSE: 1255
\end{verbatim}


To relate the quality proxy $GQ$ \eqref{eq:rev} with the vineyard age,
and to do a bootstrap re-sampling of 500 simulations 
to estimate standard errors, confidence intervals, and to obtain 
Figure~\ref{fig:2}, we have used 
\textsf{Python 3.5.3} with the following code:
\begin{verbatim}
import pandas as pd
import numpy as np
from cairocffi import *
import matplotlib
matplotlib.use('Gtk3Agg')
import matplotlib.pyplot as plt
df = pd.read_excel('___.xlsx')   #UPLOAD THE DATABASE FILE
age = df['age']
quality = df['quality']
from sklearn.linear_model import LinearRegression
X = np.vstack([age, np.ones(len(age))]).T
plt.figure(figsize=(12,8))
for i in range(0,500):
sample_index = np.random.choice(range(0, len(Bundle)), len(Bundle))
X_Samples = X[sample_index]
Y_Samples = Bundle[sample_index]
lr = LinearRegression()
lr.fit(X_Samples,Y_Samples)
plt.plot(Idade, lr.predict(X), color='grey', alpha=0.2, zorder=1)
plt.scatter(Idade,Bundle, marker='o', color='orchid', zorder=4 )

lr = LinearRegression()
lr.fit(X, Bundle)
LinearRegression(copy_X=True, fit_intercept=True, n_jobs=1, normalize=False)
plt.plot(Idade, lr.predict(X), color='red', zorder=5)
plt.xlabel('Age')
plt.ylabel('Quality Proxy')
plt.savefig("/home/pi/Desktop/Anibal_Docs/scatter3.png", dpi=125)

import statsmodels.formula.api as smf
results = smf.ols('Bundle ~ Idade',data=df).fit()
print(results.summary())

OLS Regression Results                            
==============================================================================
Dep. Variable:                 Bundle   R-squared:                       0.051
Model:                            OLS   Adj. R-squared:                 -0.054
Method:                 Least Squares   F-statistic:                    0.4849
Date:                Thu, 12 Apr 2018   Prob (F-statistic):              0.504
Time:                        09:51:11   Log-Likelihood:                 8.0719
No. Observations:                  11   AIC:                            -12.14
Df Residuals:                       9   BIC:                            -11.35
Df Model:                           1                                         
Covariance Type:            nonrobust                                         
==============================================================================
coef    std err          t      P>|t|      [0.025    0.975]
------------------------------------------------------------------------------
Intercept      0.5858      0.078      7.468      0.000       0.408       0.763
Idade          0.0036      0.005      0.696      0.504      -0.008       0.015
==============================================================================
Omnibus:                        2.482   Durbin-Watson:                   2.092
Prob(Omnibus):                  0.289   Jarque-Bera (JB):                1.450
Skew:                          -0.651   Prob(JB):                        0.484
Kurtosis:                       1.788   Cond. No.                         31.0
==============================================================================
\end{verbatim}


The numerical simulations reported in Section ``Results and Discussion''
were performed in the \textsf{Xpress-Mosel} multi-solver modeling 
and problem solving environment. 
Follows our code for the first simulation of the producer 
with a time-frame of 5 years (the code for the other
simulations is similar):
\begin{verbatim}
model OptimizationVine
uses "mmxnlp"

declarations
nT=4 		!n year planning
T=0..nT		!planing horizon 
nd=5   	    !total number of plots
P=1..nd 	!plot set
qc=0.0036   !quality constant
p0=-661.4   !quantity constant 
p1=451.1
p2=-6.774
Pu=3        !grape selling price (kg)
S=10000     !vineyard replacement cost (per hectare)

Ii: array(P) of integer
A: array(P) of real

u: array(P,T) of mpvar
Id: array(P,T) of mpvar

end-declarations

Ii::[20,30,11,5,58] !Each plot age at 2017 (or T=0)
A::[4.47,1.45,0.44,1.6,0.5] !Each plot area

forall(j in P,t in T) do
create(u(j,t))    				
u(j,t) is_binary
end-do

forall(j in P,t in T) do
create(Id(j,t))    				
Id(j,t) is_integer
end-do

!ORIGINAL
objective:= sum(t in T,j in P) 
(qc*Id(j,t)*A(j)*Pu*(p2*(Id(j,t))^2+p1*Id(j,t)+p0)-S*A(j)*u(j,t))

!constraints

forall(j in P) do
sum(t in T) u(j,t) <= 1
end-do

forall(j in P,t in T) do
Id(j,t)=t+Ii(j)-sum(l in 0..t-1) (u(j,l)*(Id(j,l)+1))    	
end-do

setparam("XPRS_verbose", true);
maximize(objective);

forall(j in P,t in T) do
if(getsol(u(j,t))<>0)
then writeln("Plot ",j," is replaced at period ",t," u(",j,",",t,")=",getsol(u(j,t)))
end-if 	!writes only if a plot is replaced
end-do 

forall(j in P,t in T) do
writeln("Plot age ",j," at period ",t," ID(",j,",",t,")=",getsol(Id(j,t)))
end-do                   !retrieves the age of each plot per period

forall(t in T) do
writeln("Yield of period ",t," =", sum(j in P) 
(qc*getsol(Id(j,t))*A(j)*Pu*(p2*(getsol(Id(j,t)))^2
+p1*getsol(Id(j,t))+p0)-S*A(j)*getsol(u(j,t))))    
end-do       !farm yield per period

writeln("Total Profit = ", getobjval, "\n")  !overall profit

end-model
\end{verbatim}



\end{document}